\documentclass[oneside]{amsart}

\usepackage{amssymb}

\newtheorem{theorem}{Theorem}[section]
\newtheorem{lemma}[theorem]{Lemma}

\theoremstyle{definition}
\newtheorem{definition}[theorem]{Definition}
\newtheorem{proposition}[theorem]{Proposition}
\newtheorem{corollary}[theorem]{Corollary}
\newtheorem{Example}[theorem]{Example}
\newtheorem{Q}[theorem]{Question}

\theoremstyle{remark}
\newtheorem{remark}[theorem]{Remark}

\numberwithin{equation}{section}

\begin{document}

\title[Proper holomorphic maps]{On proper holomorphic maps between bounded symmetric domains}
\author{Shan Tai Chan}
\address{Department of Mathematics, The University of Hong Kong, Pokfulam Road, Hong Kong}
\curraddr{}
\email{mastchan@hku.hk}
\thanks{}


\dedicatory{}

\begin{abstract}
We study proper holomorphic maps between bounded symmetric domains $D$ and $\Omega$.
In particular, when $D$ and $\Omega$ are of the same rank $\ge 2$ such that all irreducible factors of $D$ are of rank $\ge 2$, we prove that any proper holomorphic map from $D$ to $\Omega$ is a totally geodesic holomorphic isometric embedding with respect to certain canonical K\"ahler metrics of $D$ and $\Omega$.
We also obtain some results regarding holomorphic maps $F:D\to \Omega$ which map minimal disks of $D$ properly into rank-$1$ characteristic symmetric subspaces of $\Omega$.
On the other hand, we obtain new rigidity results regarding semi-product proper holomorphic maps between $D$ and $\Omega$ under a certain rank condition on $D$ and $\Omega$.
\end{abstract}

\maketitle
\bibliographystyle{amsalpha}

\section{Introduction}
In \cite{Ts93}, Tsai has proven that if $F$ is a proper holomorphic map from an irreducible bounded symmetric domain $D$ to a bounded symmetric domain $\Omega$ with the assumption that $\mathrm{rank}(D)\ge \mathrm{rank}(\Omega) \ge 2$, then $\mathrm{rank}(D)= \mathrm{rank}(\Omega)$ and $F$ is a totally geodesic holomorphic isometric embedding with respect to the Bergman metrics up to a normalizing constant.
In general, a proper holomorphic map $f$ between reducible bounded symmetric domains $D_1$ and $D_2$ of equal rank $\ge 2$ can be nonstandard (i.e., not totally geodesic) when the domain $D_1$ of $f$ is reducible and has a rank-$1$ irreducible factor.
We will give an example of such a proper holomorphic map. This example will also allow us to formulate an appropriate rigidity theorem (i.e., Theorem \ref{thm:TrIrr2}) for proper holomorphic maps between reducible bounded symmetric domains.
Let $D$ and $\Omega$ be irreducible bounded symmetric domains of rank $\ge 2$.
In \cite{Ng15}, Ng has proven that if a holomorphic map $f:D\to \Omega$ maps minimal disks of $D$ properly into the rank-$1$ characteristic symmetric subspaces of $\Omega$, then $f$ is a totally geodesic holomorphic isometric embedding with respect to the Bergman metrics up to a normalizing constant.

In the first part of this article, we will study proper holomorphic maps between (reducible) bounded symmetric domains along the lines of Ng \cite{Ng15}.
For an irreducible bounded symmetric domain $U\Subset \mathbb C^n$, we let $g_U$ be the canonical K\"ahler-Einstein metric on $U$ normalized so that minimal disks of $U$ are of constant Gaussian curvature $-2$, and we denote by $\omega_{g_U}$ the K\"ahler form of $(U,g_U)$.
Then, $g_U$ agrees with the standard complex Euclidean metric of $\mathbb C^n$ at ${\bf 0}$.
For K\"ahler manifolds $(M,g_M)$ and $(N,g_N)$ with the corresponding K\"ahler forms $\omega_{g_M}$ and $\omega_{g_N}$ respectively, a holomorphic map $F:(M,g_M)\to (N,g_N)$ is said to be \emph{isometric} if and only if $F^*\omega_{g_N}=\omega_{g_M}$.
In addition, a holomorphic map $F:(M,g_M)\to (N,g_N)$ is said to be \emph{an isometric map up to a normalizing constant} if and only if $F^*\omega_{g_N}=\lambda \omega_{g_M}$ for some positive real constant $\lambda$.
Motivated by \cite[Proposition 1.2]{Ng15}, we will also study holomorphic maps $F:D\to \Omega$ between (reducible) bounded symmetric domains $D$ and $\Omega$ which map minimal disks of $D$ properly into rank-$1$ characteristic symmetric subspaces of $\Omega$.
In this direction, we have the following generalization of \cite[Proposition 1.2]{Ng15}.

\begin{theorem}\label{thm:TrIrr1}
Let $F:D\to \Omega$ be a holomorphic map between bounded symmetric domains $D$ and $\Omega$ such that $F$ maps the minimal disks of $D$ properly into rank-$1$ characteristic symmetric subspaces of $\Omega$.
Suppose all irreducible factors of $D$ are of rank at least two.
Write $D=D_1\times\cdots \times D_k$ and $\Omega=\Omega_1\times\cdots \times \Omega_l$, where $D_j\Subset \mathbb C^{n_j}$, $1\le j\le k$, and $\Omega_j$, $1\le j\le l$, are the irreducible factors of $D$ and $\Omega$ respectively with $\mathrm{rank}(D_j)\ge 2$ for $1\le j\le k$.
Then, $F$ is a totally geodesic isometric embedding from $(D_1,g_{D_1})\times\cdots \times (D_k,g_{D_k})$ to $(\Omega_1,g_{\Omega_1})\times\cdots \times (\Omega_l,g_{\Omega_l})$.
\end{theorem}

In the consideration of proper holomorphic maps between bounded symmetric domains, we will deduce the following result from Theorem \ref{thm:TrIrr1}.

\begin{theorem}\label{thm:TrIrr2}
Let $F:D\to \Omega$ be a proper holomorphic map between bounded symmetric domains $D$ and $\Omega$ such that $\mathrm{rank}(D)=\mathrm{rank}(\Omega)$.
Suppose all irreducible factors of $D$ are of rank at least two.
Write $D=D_1\times\cdots \times D_k$ and $\Omega=\Omega_1\times\cdots \times \Omega_l$, where $D_j$, $1\le j\le k$, and $\Omega_j$, $1\le j\le l$, are the irreducible factors of $D$ and $\Omega$ respectively with $\mathrm{rank}(D_j)\ge 2$ for $1\le j\le k$.
Then, $F$ is a totally geodesic isometric embedding from $(D_1,g_{D_1})\times\cdots \times (D_k,g_{D_k})$ to $(\Omega_1,g_{\Omega_1})\times\cdots \times (\Omega_l,g_{\Omega_l})$.
\end{theorem}

In \cite{Seo18}, Seo has introduced semi-product proper holomorphic maps between (reducible) bounded symmetric domains.
Then, Seo \cite{Seo18} has proven that any proper rational map between (reducible) bounded symmetric domains is a semi-product proper holomorphic map. One of the main results in Seo \cite{Seo18} is the classification of all proper holomorphic maps between (reducible) bounded symmetric domains of the same dimension (see \cite[Theorem 1.2]{Seo18}).
Motivated by the work of Seo \cite{Seo18}, we will study semi-product proper holomorphic maps between non-equidimensional (reducible) bounded symmetric domains. Under certain rank conditions, we are able to get the complete description for such maps (see Theorem \ref{thm:Structure_PHM1}).

\section{Preliminaries}
For a (reducible) bounded symmetric domain $D=D_1\times\cdots \times D_k$, where $D_j$, $1\le j\le k$, are the irreducible factors of $D$, there is a K\"ahler metric $g'_D$ on $D$ such that $(D,g'_D)\cong(D_1,\lambda_1g_{D_1})\times\cdots \times (D_k,\lambda_kg_{D_k})$ for some positive real constants $\lambda_j$, $1\le j\le k$, namely,
\[ g'_D = \sum_{j=1}^k \lambda_j \pi_j^*g_{D_j}, \]
where $\pi_j:D\to D_j$ is the canonical projection onto the $j$-th irreducible factor of $D$, $\pi_j(Z^1,\ldots,Z^k)=Z^j$ for $(Z^1,\ldots,Z^k)\in D_1\times\cdots \times D_k=D$, $1\le j\le k$.
In what follows, for any bounded symmetric domain $D$ we call such a metric $g'_D$ on $D$ a \textbf{canonical K\"ahler metric} and denote by $\mathrm{rank}(D)$ the rank of $D$.
It is well-known that a bounded symmetric domain $D$ is of rank $1$ if and only if $D$ is biholomorphic to a complex unit ball.

Denote by $\Delta^k=\{(z_1,\ldots,z_k)\in \mathbb C^k: |z_j|<1,\;1\le j\le k\}$ the $k$-disk in $\mathbb C^k$ for any integer $k\ge 1$.
We let $\mathbb B^n$ be the complex unit ball in the complex $n$-dimensional Euclidean space $\mathbb C^n$ with respect to the standard complex Euclidean metric, i.e.,
\[ \mathbb B^n:=\left\{(z_1,\ldots,z_n)\in \mathbb C^n: \sum_{j=1}^n |z_j|^2 < 1 \right\}.\]
For any complex manifold $M$, we denote by $T^{1,0}_x(M)=T_x(M)$ the holomorphic tangent space to $M$ at $x\in M$.

Let $D\cong G/K$ be a bounded symmetric domain in $\mathbb C^n$ and let $r:=\mathrm{rank}(D)$, where $G$ is the identity component of the automorphism group of $D$ and $K\subset G$ is the isotropy subgroup of $G$ at ${\bf 0}\in \mathbb C^n$. By the Polydisk Theorem (cf.\,\cite[p.\,88]{Mok89}, \cite{Wo72}), there is a totally geodesic complex submanifold $\Pi\cong \Delta^r$ of $D$ such that
\[ D=\bigcup_{k\in K} k\cdot \Pi. \]
A vector $v\in T_x(D)$, $x\in D$, is said to be a \textbf{characteristic vector} of $D$ at $x$ if $v$ is tangent to any direct factor of a totally geodesic $r$-disk of $D$ (cf.\,\cite[Section 2]{Ng15}).
Write $D=D_1\times\cdots \times D_k$, where $D_j$, $1\le j\le k$, are the irreducible factors of $D$.
Then, it follows from Wolf \cite{Wo72} that any rank-$1$ characteristic symmetric subspace of $D$ is of the form
$\{x_1\}\times\cdots \times \{x_{j-1}\} \times B_j \times \{x_{j+1}\}\times\cdots \times \{x_k\}$ for some $j$, $1\le j\le k$, where $B_j\cong \mathbb B^{m_j}$ is a rank-$1$ characteristic symmetric subspace of $D_j$, $x_\mu \in D_\mu$ is a point for each $\mu\neq j$, and $m_j$ is a positive integer depending on $D_j$.
Here, we also know that $(B_j,g_{D_j}|_{B_j})$ is holomorphically isometric to $(\mathbb B^{m_j},g_{\mathbb B^{m_j}})$, which is of constant holomorphic sectional curvature $-2$. For the notion of characteristic symmetric subspaces of bounded symmetric domains, we refer the readers to Mok-Tsai \cite{MT92}.

\section{Proper holomorphic maps between bounded symmetric domains of equal rank $\ge 2$}
Motivated by the study in Tsai \cite{Ts93} and Ng \cite{Ng15}, we are concerning proper holomorphic maps between (reducible) bounded symmetric domains of the same rank $\ge 2$.
In \cite{Ts93}, Tsai has proven that if $F:D\to \Omega$ is a proper holomorphic map between bounded symmetric domains $D$ and $\Omega$, then $\mathrm{rank}(D)\le \mathrm{rank}(\Omega)$.
Thus, it is natural to ask the following question.
\begin{Q}\label{Pro_RP1}
Let $F:D\to \Omega$ be a proper holomorphic map between bounded symmetric domains $D$ and $\Omega$.
If $\mathrm{rank}(D)=\mathrm{rank}(\Omega)\ge 2$, then is $F$ a totally geodesic holomorphic isometric embedding with respect to some canonical K\"ahler metrics on $D$ and $\Omega$?
\end{Q}
\begin{remark}
Tsai \cite[Main Theorem]{Ts93} has an affirmative answer to Question \ref{Pro_RP1} under the assumption that $D$ is irreducible.
\end{remark}

However, we have a negative answer to Question \ref{Pro_RP1} if $D$ is reducible and some irreducible factor of the domain $D$ is a complex unit ball, namely, we have
\begin{Example}\label{Eg:1}
We also denote by $M(p,q;\mathbb C)$ the space of $p$-by-$q$ complex matrices.
A type-$\mathrm{I}$ irreducible bounded symmetric domain is given by
\[ D^{\mathrm{I}}_{p,q}:=\left\{ Z \in M(p,q;\mathbb C): {\bf I}_q - \overline{Z}^t Z > 0 \right\}, \]
where $p$ and $q$ are positive integers. We refer the readers to Mok \cite{Mok89} for details about bounded symmetric domains.

For any integer $n\ge 2$, it is well-known that there is a positive integer $q$ and a proper holomorphic map $f:\mathbb B^n\to\mathbb B^{q-2}$ which is not a holomorphic isometry from $(\mathbb B^n,\lambda g_{\mathbb B^n})$ to $(\mathbb B^{q-2}, g_{\mathbb B^{q-2}})$ for any real constant $\lambda>0$ $($cf.\, D'Angelo \cite{D88}$)$.
More precisely, from D'Angelo \cite[p.\,84]{D88} we may let $q=2n+1$ and
\[ f(z_1,\ldots,z_n):=(z_1,\ldots,z_{n-1},z_1z_n,z_2z_n,\ldots,z_{n-1}z_n,z_n^2). \]
Writing $f=(f_1,\ldots,f_{q-2})$, we define a map $F:\mathbb B^n \times D^{\mathrm{I}}_{2,2} \to D^{\mathrm{I}}_{3,q}$ by
\[ F(z_1,\ldots,z_n;{\bf W})
=\begin{pmatrix}
f_1(z_1,\ldots,z_n) &\cdots & f_{q-2}(z_1,\ldots,z_n) & {\bf 0}\\
0 &\cdots & 0 & {\bf W}
\end{pmatrix} \]
for $(z_1,\ldots,z_n)\in \mathbb B^n$ and ${\bf W}\in D^{\mathrm{I}}_{2,2}$.
Then, $F$ is a proper holomorphic map between the bounded symmetric domains $\mathbb B^n \times D^{\mathrm{I}}_{2,2}$ and $D^{\mathrm{I}}_{3,q}$ of rank three such that $F$ is not a holomorphic isometry with respect to any canonical K\"ahler metrics of $\mathbb B^n \times D^{\mathrm{I}}_{2,2}$ and $D^{\mathrm{I}}_{3,q}$.
\end{Example}

\begin{remark}
\text{}
\begin{enumerate}
\item It is known from Chan-Xiao-Yuan \cite{CXY17} and Mok \cite{Mok12} that any holomorphic isometry between bounded symmetric domains with respect to the canonical K\"ahler metrics is a proper holomorphic map.
In \cite{Ch19}, we have shown that any holomorphic isometry between bounded symmetric domains of the same rank with respect to the canonical K\"ahler metrics is totally geodesic.
From Example \ref{Eg:1}, we know that this result from \cite{Ch19} cannot be generalized to the case of proper holomorphic maps unless we impose additional assumptions on the bounded symmetric domains.
\item Example \ref{Eg:1} shows that Theorems \ref{thm:TrIrr1} and \ref{thm:TrIrr2} cannot be generalized to the case where some irreducible factor of the domain $D$ is of rank $1$.
\end{enumerate}
\end{remark}

We first recall the following lemma obtained from Mok-Tsai \cite{MT92} and Tsai \cite{Ts93}, which is known by Ng \cite[p.\,224]{Ng15}.
\begin{lemma}[cf.\,Mok-Tsai \cite{MT92}, Tsai \cite{Ts93} and Ng \cite{Ng15}]\label{lem_Preserve_rk1CSS}
Let $F:D\to \Omega$ be a proper holomorphic map between bounded symmetric domains $D$ and $\Omega$.
Suppose $\mathrm{rank}(D)=\mathrm{rank}(\Omega) \ge 2$.
Then, $F$ maps rank-$1$ characteristic symmetric subspaces of $D$ properly into rank-$1$ characteristic symmetric subspaces of $\Omega$.
In particular, $F$ maps minimal disks of $D$ properly into rank-$1$ characteristic symmetric subspaces of $\Omega$.
\end{lemma}
This yields the following obvious corollary.
\begin{corollary}
Let $F:D\to \Omega$ be a proper holomorphic map between bounded symmetric domains $D$ and $\Omega$ such that $\mathrm{rank}(D)=\mathrm{rank}(\Omega)$.
If $\Omega$ is of tube type, then so is $D$.
\end{corollary}
\begin{proof}
Suppose $\mathrm{rank}(D)=\mathrm{rank}(\Omega)=1$. Then, $D$ (resp.\,$\Omega$) is biholomorphic to a complex unit ball. Since $\Omega$ is of tube type, $\Omega$ is the complex unit disk and thus $D$ can only be the complex unit disk as well. In particular, $D$ is of tube type.

Now, we suppose $\mathrm{rank}(D)=\mathrm{rank}(\Omega)\ge 2$.
Note that rank-$1$ characteristic symmetric subspaces of $\Omega$ are precisely the minimal disks of $\Omega$ because $\Omega$ is of tube type (cf.\,Mok-Tsai \cite{MT92} and Wolf \cite{Wo72}).
By Lemma \ref{lem_Preserve_rk1CSS}, $F$ maps rank-$1$ characteristic symmetric subspaces of $D$ properly into rank-$1$ characteristic symmetric subspaces of $\Omega$.
Therefore, rank-$1$ characteristic symmetric subspaces of $D$ could only be unit disks.
Hence, all irreducible factors of $D$ are of tube type and so is $D$ by Wolf \cite{Wo72}.
\end{proof}

We observe that \cite[Proposition 1.2]{Ng15} actually holds by \cite[Proof of Proposition 1.2]{Ng15} even when the target bounded symmetric domain is reducible, namely, we have

\begin{proposition}[cf.\,Proposition 1.2 in Ng \cite{Ng15}]\label{Pro:Pro1.2_Ng15}
Let $D$ and $\Omega$ be bounded symmetric domains of rank $\ge 2$ and let $F:D\to \Omega$ be a holomorphic map.
Suppose $D$ is irreducible and $F$ maps the minimal disks of $D$ properly into the rank-$1$ characteristic symmetric subspaces of $\Omega$.
Write $\Omega=\Omega_1\times\cdots \times \Omega_l$, where $\Omega_j$, $1\le j\le l$, are the irreducible factors of $\Omega$.
Then, $F$ is a totally geodesic isometric embedding from $(D,g_D)$ to $(\Omega_1,g_{\Omega_1})\times\cdots \times (\Omega_l,g_{\Omega_l})$.
\end{proposition}
\begin{remark}
From the proof of Proposition 1.2 in Ng \cite{Ng15}, we know that $F$ is a totally geodesic isometric embedding from $(D,\lambda g_D)$ to $(\Omega_1,g_{\Omega_1})\times\cdots \times (\Omega_l,g_{\Omega_l})$ for some positive real constant $\lambda$.
But then by the fact that $F$ maps minimal disks of $D$ properly into the rank-$1$ characteristic symmetric subspaces of $\Omega$ and $F$ is totally geodesic, we can deduce that $\lambda=1$.
\end{remark}

By making use of Proposition \ref{Pro:Pro1.2_Ng15} and results in Ng \cite{Ng15}, we are ready to prove Theorem \ref{thm:TrIrr1}, as follows.

\begin{proof}[Proof of Theorem \ref{thm:TrIrr1}]
We write $Z^j=(Z^j_1,\ldots,Z^j_{n_j})\in D_j \Subset \mathbb C^{n_j}$ for the Harish-Chandra coordinates of $D_j$, $1\le j\le k$.
For $W = (W^1,\ldots,W^k)\in D$, we let 
$\iota_{j,W}:D_j\hookrightarrow D$ be the natural embedding given by 
\[ \iota_{j,W}(Z^j):=(W^1,\ldots,W^{j-1},Z^j,W^{j+1},\ldots,W^k)\]
for $Z^j \in D_j$, $1\le j\le k$.
Then, each $F\circ\iota_{j,W}:D_j\to \Omega$ is a holomorphic map which maps the minimal disks of $D_j$ properly into rank-$1$ characteristic symmetric subspaces of $\Omega$, $1\le j\le k$.
Since $D_j$ is an irreducible bounded symmetric domain of rank $\ge 2$, it follows from Proposition \ref{Pro:Pro1.2_Ng15} that $F\circ\iota_{j,W}$ is a totally geodesic holomorphic isometric embedding from $(D_j,g_{D_j})$ to $(\Omega_1,g_{\Omega_1})\times\cdots \times (\Omega_l,g_{\Omega_l})$, $1\le j\le k$.
Let $g'_\Omega$ be the canonical K\"ahler metric on $\Omega$ such that $(\Omega,g'_\Omega)\cong (\Omega_1,g_{\Omega_1})\times\cdots \times (\Omega_l,g_{\Omega_l})$.
Therefore, for any $W\in D$ we have $\iota_{j,W}^*(F^*\omega_{g'_\Omega})=\omega_{g_{D_j}}$ for $1\le j\le k$.

Write $h:=F^*g'_\Omega$ and $\omega_h:=F^*\omega_{g'_\Omega}$.
For $1\le j\le k$, let $\pi_j:D\to D_j$ be the canonical projection onto the $j$-th factor, i.e., $\pi_j(W^1,\ldots,W^k)=W^j$ for $(W^1,\ldots,W^k)\in D$.
Let $Z=(Z^1,\ldots,Z^k)\in D$.
Note that $T_Z(D)=T_{Z^1}(D_1)\oplus \cdots \oplus T_{Z^k}(D_k)$.
For any $v\in T_Z(D)$, we write $v = v_1+\ldots+v_k$, where $v_j\in T_{Z^j}(D_j)$ for $1\le j\le k$.
Furthermore, for $1\le j\le k$ and for any tangent vector $v_j\in T_{Z^j}(D_j)$ we may write $v_j=\sum_{\mu=1}^{r_j} v_{j,\mu} e^{(j)}_\mu$ in normal form (cf.\,Mok \cite[p.\,252]{Mok89}), where $r_j:=\mathrm{rank}(D_j) \ge 2$ and $\{e^{(j)}_\mu\}_{\mu=1}^{r_j}$ is the standard basis for the holomorphic tangent space of a totally geodesic $r_j$-disk of $D_j$ through the point $Z^j\in D_j$. In this situation, $e^{(j)}_\mu$, $1\le \mu\le r_j$, are characteristic vectors of $T_{Z^j}(D_j)$, for $1\le j\le k$.
From \cite[Proof of Lemma 3.1]{Ng15} we have
\[ h\big(e^{(i)}_\mu,\overline{e^{(j)}_\nu}\big) = 0 \]
for distinct $i,j$, $1\le i,j\le k$, and for any $\mu,\nu$, $1\le \mu\le r_i$, $1\le \nu \le r_j$. In general, letting $\alpha_\mu \in T_{Z^\mu}(D_\mu) \subset T_Z(D)$ be characteristic vectors, $1\le \mu\le k$, we have $h(\alpha_i,\overline{\alpha_j}) = 0$ for distinct $i,j$, $1\le i,j\le k$.
This implies that for tangent vectors $v_\mu\in T_{Z^\mu}(D_\mu) \subset T_Z(D)$, $1\le \mu\le k$, we have $h(v_i,\overline{v_j})=0$ for distinct $i,j$, $1\le i,j\le k$.
In particular, we have
\[ \omega_h = \sqrt{-1}\sum_{j=1}^k \sum_{1\le \mu,\nu\le n_j} h^{(j)}_{\mu\overline \nu}(Z) dZ^j_\mu \wedge d\overline{Z^j_\nu} \]
on $D$.
Recall that for $W\in D$ we have $\iota_{i,W}^*\omega_h = \omega_{g_{D_i}}$ for $1\le i\le k$. Thus, for $1\le j\le k$, each $h^{(j)}_{\mu\overline \nu}(Z)$, $1\le \mu,\nu\le n_j$, only depends on $Z^j$, i.e., $h^{(j)}_{\mu\overline \nu}(Z)\equiv h^{(j)}_{\mu\overline \nu}(Z^j)$.
In addition, for $1\le j\le k$ we have 
\[ \sqrt{-1}\sum_{1\le \mu,\nu\le n_j} h^{(j)}_{\mu\overline \nu}(Z^j) dZ^j_\mu \wedge d\overline{Z^j_\nu}=\pi_j^*\omega_{g_{D_j}}\]
by $\pi_j^*\omega_{g_{D_j}}=\pi_j^*(\iota_{j,W}^*\omega_h)=(\iota_{j,W}\circ \pi_j)^*\omega_h$.
Then, we have $\omega_h = \sum_{j=1}^k \pi_j^* \omega_{g_{D_j}}$ and thus $F^*g'_\Omega=h=\sum_{j=1}^k \pi_j^* g_{D_j}$.
Hence, $F$ is a (proper) holomorphic isometric embedding from $(D,\sum_{j=1}^k \pi_j^* g_{D_j})$ to $(\Omega,g'_\Omega)$.
Since the irreducible factors of $D$ are of rank $\ge 2$, it follows from the arguments of \cite[Proof of Theorem 1.3.2]{Mok12} that the second fundamental form of $(F(D),g'_\Omega|_{F(D)})$ in $(\Omega,g'_\Omega)$ vanishes identically and thus $F$ is totally geodesic, as desired.
\end{proof}

As a consequence, we have a simple proof of Theorem \ref{thm:TrIrr2} in the following.
(Noting that Theorem \ref{thm:TrIrr2} actually provides an affirmative answer to Question \ref{Pro_RP1} under the assumption that all irreducible factors of the domain $D$ are of rank $\ge 2$.)

\begin{proof}[Proof of Theorem \ref{thm:TrIrr2}]
By Lemma \ref{lem_Preserve_rk1CSS}, $F$ maps minimal disks of $D$ properly into rank-$1$ characteristic symmetric subspaces of $\Omega$.
Then, the result follows from Theorem \ref{thm:TrIrr1}.
\end{proof}

Now, we study holomorphic maps $f:D\to \Omega$ which map minimal disks of $D$ properly into rank-$1$ characteristic symmetric subspaces of $\Omega$, where $D$ and $\Omega$ are bounded symmetric domains such that $\Omega$ is reducible. The case where the reducible bounded symmetric domain $\Omega$ has an irreducible factor of rank $\ge 2$ can be quite complicated in general if some irreducible factors of the domain $D$ are complex unit balls (See Example \ref{Eg:1}). Therefore, we will focus on the simplest case where the target $\Omega$ is a product of complex unit balls.
We first recall a result of Ng \cite{Ng15}.
\begin{lemma}[cf.\,Proposition 2.3 in \cite{Ng15}]\label{lem:Pro2.3_Ng15}
Let $F:\Delta \times U \to \mathbb B^n$ be a holomorphic map such that $F|_{\Delta\times\{{\bf 0}\}}: \Delta \cong \Delta\times\{{\bf 0}\} \to \mathbb B^k$ is a proper map, where $U\Subset \mathbb C^m$ is a bounded domain containing ${\bf 0}$. Then, for any $(z,w)\in \Delta\times U$ we have $F(z,w)=F(z,{\bf 0})$.
\end{lemma}

On the other hand, by Mok \cite{Mok16} and Yuan-Zhang \cite{YZ12}, we observe the non-existence of holomorphic isometries between certain bounded symmetric domains with respect to the canonical K\"ahler metrics, as follows.
\begin{proposition}\label{Pro:NonEx_HI_BSD_to_PBs}
Let $\Omega\Subset \mathbb C^N$ be a bounded symmetric domain such that $\Omega$ has an irreducible factor of rank $\ge 2$, i.e., $\Omega=\Omega_1\times\cdots \times \Omega_n$ and there exists $j$, $1\le j\le n$, such that $\mathrm{rank}(\Omega_j)\ge 2$, where $\Omega_i$, $1\le i\le n$, are the irreducible factors of $\Omega$.
Equip a K\"ahler metric $g'_\Omega$ on $\Omega$ so that
$(\Omega,g'_\Omega)\cong (\Omega_1,\lambda_1g_{\Omega_1})\times\cdots \times (\Omega_n,\lambda_ng_{\Omega_n})$ for some positive real constants $\lambda_j$, $1\le j\le n$.
Then, there does not exist a holomorphic isometry from $(\Omega,g'_\Omega)$ to $(\mathbb B^{N_1},\mu_1 g_{\mathbb B^{N_1}})\times\cdots \times (\mathbb B^{N_m},\mu_m g_{\mathbb B^{N_m}})$, where $\mu_l$, $1\le l\le m$, are positive real constants.
\end{proposition}
\begin{proof}
Assume the contrary that there exists a holomorphic isometry $f$ from $(\Omega,g'_\Omega)$ to $(\mathbb B^{N_1},\mu_1 g_{\mathbb B^{N_1}})$ $\times$ $\cdots$ $\times$ $(\mathbb B^{N_m},\mu_m g_{\mathbb B^{N_m}})$, where $\mu_l$, $1\le l\le m$, are positive real constants.
Then, by restricting to the irreducible factor $\Omega_j$ of $\Omega$, we have a holomorphic isometry $\hat f$ from $(\Omega_j,g_{\Omega_j})$ to $(\mathbb B^{N_1},\mu'_1 g_{\mathbb B^{N_1}})\times\cdots \times (\mathbb B^{N_m},\mu'_m g_{\mathbb B^{N_m}})$, where $\mu'_l:={\mu_l\over \lambda_j}$ for $1\le l\le m$.
Write $\Omega_j$ for an irreducible factor of $\Omega$ such that $\mathrm{rank}(\Omega_j)\ge 2$.
Then, it follows from \cite{Mok16} that there exists a nonstandard (i.e., not totally geodesic) holomorphic isometry $F$ from $(\mathbb B^k,g_{\mathbb B^k})$ to $(\Omega_j,g_{\Omega_j})$ for some integer $k\ge 2$.
This gives a holomorphic isometry $\hat f\circ F$ from $(\mathbb B^k,g_{\mathbb B^k})$ to $(\mathbb B^{N_1},\mu'_1 g_{\mathbb B^{N_1}})\times\cdots \times (\mathbb B^{N_m},\mu'_m g_{\mathbb B^{N_m}})$.
By the rigidity theorem of Yuan-Zhang \cite{YZ12}, $\hat f\circ F$ is totally geodesic.
This contradicts with the fact that $F$ is not totally geodesic.
Hence, there does not exist such a holomorphic isometry $f$, as desired.
\end{proof}

Now, by making use of the technique in Ng \cite{Ng15}, we have the following structure theorem for holomorphic maps from a bounded symmetric domain $D$ to a product $\Omega$ of complex unit balls which map minimal disks of $D$ properly into rank-$1$ characteristic symmetric subspaces of $\Omega$.
\begin{theorem}\label{thm:H_map_MD_to_rk1CSS_1}
Let $D=D_1\times\cdots \times D_k$ be a bounded symmetric domain of rank $\ge 2$ and $\Omega:=\mathbb B^{m_1}\times\cdots \times \mathbb B^{m_l}$ be a product of complex unit balls, where $D_i$, $1\le i\le k$, are irreducible bounded symmetric domains. 
Let $f:D\to \Omega$ be a holomorphic map which maps minimal disks of $D$ properly into rank-$1$ characteristic symmetric subspaces of $\Omega$.
Write $f=(f_1,\ldots,f_l)$, where $f_j:D\to \mathbb B^{m_j}$, $1\le j\le l$, are holomorphic maps.
Then, we have $k\le l$ and up to a permutation of the irreducible factors of $\Omega$, we have
\[ f(Z^1,\ldots,Z^k)
=\begin{cases}(f_1(Z^1),\ldots,f_k(Z^k),f_{k+1}(Z),\ldots,f_l(Z)) &\text{ if } k< l, \\ (f_1(Z^1),\ldots,f_k(Z^k)) & \text{ if } k=l, 
\end{cases}\]
for $Z=(Z^1,\ldots,Z^k)\in D_1\times\cdots \times D_k=D$.
Moreover, $D_i\cong \mathbb B^{n_i}$ is a complex unit ball for some $n_i\ge 1$, $1\le i\le k$, i.e., $D$ is also a product of complex unit balls.
\end{theorem}
\begin{proof}
We may assume without loss of generality that $f({\bf 0})={\bf 0}$. For each $i$, $1\le i\le k$, we choose a minimal disk $\Delta^{(i)}:=\{(z,0,\ldots,0)\in D_i: |z|<1\}\subseteq D_i$. Write $M_i:=D_1\times\cdots\times D_{i-1}\times \Delta^{(i)} \times D_{i+1}\times\cdots \times D_k \subseteq D$ for $1\le i\le k$. Restricting $f$ to the minimal disk $\hat\Delta^{(i)}:=\{(0,\ldots,0)\} \times \Delta^{(i)} \times \{(0,\ldots,0)\}\subseteq M_i \subseteq D$, we have $f(\hat\Delta^{(i)})\subseteq B_i \subset \Omega$ for some rank-$1$ characteristic symmetric subspace $B_i$ of $\Omega$ which contains ${\bf 0}$. Note that such a rank-$1$ characteristic symmetric subspace $B_i$ is exactly $\{(0,\ldots,0)\}\times \mathbb B^{m_{j_i}} \times \{(0,\ldots,0)\}$ for some $j_i$, $1\le j_i\le l$. Thus, $f_j(\hat\Delta^{(i)})=\{{\bf 0}\}$ for $j\neq j_i$, and $f_{j_i}|_{\hat\Delta^{(i)}}:\hat\Delta^{(i)}\to B_i \cong \mathbb B^{m_{j_i}}$ is a proper holomorphic map. We write $Z=(Z^1,\ldots,Z^k)\in D_1\times\cdots \times D_k=D$ and $Z^i\in D_i$ is the Harish-Chandra coordinates of $D_i$, $1\le i\le k$. By \cite[Proposition 2.3]{Ng15} (i.e.,\,Lemma \ref{lem:Pro2.3_Ng15}), we have
\[ f_{j_i}|_{M_i}(Z^1,\ldots,Z^{i-1};z,0,\ldots,0;Z^{i+1},\ldots,Z^k)\equiv f_{j_i}|_{M_i}(0,\ldots,0;z,0,\ldots,0;0,\ldots,0) \]
and thus $f_{j_i}(Z^1,\ldots,Z^k)\equiv f_{j_i}({\bf 0};Z^i;{\bf 0})$ by the definition of $M_i$.
In other words, $f_{j_i}$ is independent of the variables $Z^\mu$ for all $\mu\neq i$, i.e., $f_{j_i}(Z)\equiv f_{j_i}(Z^i)$.
It then follows that for distinct $i_1,i_2$, $1\le i_1,i_2\le k$, we have $j_{i_1}\neq j_{i_2}$ and thus $k\le l$.
We may assume that $j_i=i$ for $1\le i\le k$ after a permutation of the irreducible factors of $\Omega$.
Then, from the above results we have
\[ f(Z^1,\ldots,Z^k)=\begin{cases}(f_1(Z^1),\ldots,f_k(Z^k),f_{k+1}(Z),\ldots,f_l(Z)) &\text{ if } k< l, \\(f_1(Z^1),\ldots,f_k(Z^k)) & \text{ if } k=l. \end{cases}\]
Assume the contrary that $D$ has an irreducible $D_j$ which is of rank $\ge 2$, i.e., $\mathrm{rank}(D_j)\ge 2$.
Then, by restricting to $\{{\bf 0}\} \times D_j\times\{{\bf 0}\}\subset D$, we have a holomorphic map $F$ from $D_j$ to $\Omega=\mathbb B^{m_1}\times\cdots \times \mathbb B^{m_l}$ which maps minimal disks of $D_j$ properly into rank-$1$ characteristic symmetric subspaces of $\Omega$.
By Proposition \ref{Pro:Pro1.2_Ng15}, $F:D_j\to \Omega$ is a totally geodesic holomorphic isometric embedding with respect to certain canonical K\"ahler metrics on $D_j$ and $\Omega$, which contradicts with the result of Proposition \ref{Pro:NonEx_HI_BSD_to_PBs}.
Hence, all irreducible factors of $D_j$ are of rank $1$, i.e., $D_i\cong \mathbb B^{n_i}$ for some positive integer $n_i$, $1\le i\le k$, as desired.
\end{proof}

In general, for a holomorphic map $f:D \to \Omega$ between bounded symmetric domains $D$ and $\Omega$ of the same rank $\ge 2$ which maps minimal disks of $D$ properly into rank-$1$ characteristic symmetric subspaces of $\Omega$, we do not have the analogous structure theorem as in Theorem \ref{thm:H_map_MD_to_rk1CSS_1} if $\Omega$ is not a product of complex unit balls.
Actually, we have the following trivial example.
\begin{Example}\label{Ex:Product_Eq_rk_general}
Let $p_j,q_j$, $1\le j\le 4$, be positive integers such that $p_j\le q_j$ for $1\le j\le 4$.
Let $f:D^{\mathrm{I}}_{p_1,q_1}\times D^{\mathrm{I}}_{p_2,q_2} \times D^{\mathrm{I}}_{p_3,q_3}\times D^{\mathrm{I}}_{p_4,q_4} \to
D^{\mathrm{I}}_{p_1+p_2,q_1+q_2}\times D^{\mathrm{I}}_{p_3+p_4,q_3+q_4}$ be the holomorphic map defined by
\[ f(Z)
:=\left(f_1(Z),f_2(Z)\right) \]
for $Z=(Z^1,Z^2,Z^3,Z^4)\in D^{\mathrm{I}}_{p_1,q_1}\times D^{\mathrm{I}}_{p_2,q_2} \times D^{\mathrm{I}}_{p_3,q_3}\times D^{\mathrm{I}}_{p_4,q_4}$
with
\[ f_1(Z):=\begin{bmatrix}
Z^1 & {\bf 0}\\
{\bf 0} &Z^2
\end{bmatrix},\quad f_2(Z):=\begin{bmatrix}
Z^3 & {\bf 0}\\
{\bf 0} &Z^4
\end{bmatrix}. \]
It is clear that $f$ is a proper holomorphic map between bounded symmetric domains of the same rank $\sum_{j=1}^4 p_j \ge 4$ but none of the $f_1,f_2$ depends only on one of the $Z^1,\ldots,Z^4$.
\end{Example}

\section{Semi-product proper holomorphic maps between bounded symmetric domains}
Motivated by the recent work of Seo \cite{Seo18}, we will study semi-product proper holomorphic maps between (reducible) bounded symmetric domains in this section.
Let $f:D_1\times\cdots \times D_k\to \Omega_1\times\cdots \times \Omega_l$ be a proper holomorphic map, where $D_i$, $1\le i\le k$, and $\Omega_j$, $1\le j \le l$, are irreducible bounded symmetric domains.
Write $Z^j$ $($or $W^j$$)$ for the Harish-Chandra coordinates of $D_j$ for $1\le j\le k$.
In \cite{Seo18}, Seo introduced the notion of semi-product proper holomorphic maps between (reducible) bounded symmetric domains, as follows.

\begin{definition}[cf.\,Seo\,\cite{Seo18}]
The map $f$ is said to be a \textbf{semi-product proper holomorphic map} if for any $i\in \{1,\ldots,k\}$, there exists $j\in \{1,\ldots,l\}$ such that the map $f_{i,j,W}:D_i\to \Omega_j$ defined by
\[ f_{i,j,W}(Z^i) = f_j(W^1,\ldots,W^{i-1},Z^i,W^{i+1},\ldots,W^k) \]
is a proper holomorphic map for $W=(W^1,\ldots,W^{i-1},W^{i+1},\ldots,W^k)$ in some dense open subset of $D_1\times\cdots \times \widehat{D_i}\times\cdots \times D_k$.
Here, $\widehat{D_i}$ means that $D_i$ is omitted.
On the other hand, we say that the map $f$ is a \textbf{product map} if $k=l$ and 
\[ f(Z^1,\ldots,Z^k)=(f_1(Z^{\sigma(1)}),\ldots,f_k(Z^{\sigma(k)}))\] for some permutation $\sigma\in \Sigma_k$ so that each holomorphic map $f_j$ only depends on the holomorphic coordinates of $D_{\sigma(j)}$ for $1\le j\le k$.
\end{definition}
A map $F$ from a bounded domain $D\Subset \mathbb C^n$ to a bounded domain $\Omega\Subset \mathbb C^N$ is said to be \textbf{rational} if all component functions of $F$ are rational functions in $z=(z_1,\ldots,z_n)\in D\Subset \mathbb C^n$, i.e.,
$F=(F_1,\ldots,F_N)$ and $F_j(z)={P_j(z)\over Q_j(z)}$, $1\le j\le N$, for some complex polynomials $P_j,Q_j\in \mathbb C[z]$.
Then, Seo \cite{Seo18} has shown that any rational proper holomorphic map between (reducible) bounded symmetric domains is a semi-product proper holomorphic map, namely, we have
\begin{proposition}[cf. Proposition 3.5 in Seo \cite{Seo18}] \label{Pro:Pro3.5_Seo18}
Let $f:D_1\times\cdots \times D_k\to \Omega_1\times\cdots \times \Omega_l$ be a proper holomorphic map, where $D_i$, $1\le i\le k$, and $\Omega_j$, $1\le j \le l$, are irreducible bounded symmetric domains.
If $f$ is rational, then $f$ is a semi-product proper holomorphic map.
\end{proposition}

Motivated by the example of a proper holomorphic map from $D^{\mathrm{I}}_{2,2}$ to $D^{\mathrm{I}}_{3,3}$ constructed by Tsai \cite[p.\,124]{Ts93}, we give an example of a semi-product proper holomorphic map between certain reducible bounded symmetric domains which is neither a product map nor totally geodesic.
\begin{Example}\label{Eg:2}
Let $f:D^{\mathrm{I}}_{2,2}\times D^{\mathrm{I}}_{2,2}\to D^{\mathrm{I}}_{3,3} \times D^{\mathrm{I}}_{3,3}$ be a holomorphic map given by
\[ f(Z^1,Z^2)
=\left( \begin{bmatrix}
Z^1 & 0\\
0 & h_1(Z^2) g_1(Z^1)
\end{bmatrix},
\begin{bmatrix}
Z^2 & 0\\
0 & h_2(Z^1) g_2(Z^2)
\end{bmatrix}\right) \]
for $(Z^1,Z^2)\in D^{\mathrm{I}}_{2,2}\times D^{\mathrm{I}}_{2,2}$,
where $h_j$ and $g_j$ are holomorphic functions on $D^{\mathrm{I}}_{2,2}$ such that for any $W\in D^{\mathrm{I}}_{2,2}$ we have $|h_j(W)|<1$ and $|g_j(W)|<1$, $j=1,2$. Then, it is clear that $f$ is a semi-product proper holomorphic map but not a product map.
In addition, we can choose the holomorphic functions $h_j$ and $g_j$, $j=1,2$, such that $f$ is not totally geodesic.
This also shows the existence of a semi-product proper holomorphic map between bounded symmetric domains which is not a rational map.

We can actually obtain lots of holomorphic maps from $D^{\mathrm{I}}_{2,2}$ to $\Delta:=\{w\in \mathbb C: |w|<1\}$.
Write $W=\begin{pmatrix} w_{ij} \end{pmatrix}_{1\le i,j\le 2}$ and let $p(W)$ be a polynomial in $($$w_{11}$, $w_{12}$, $w_{21}$, $w_{22}$$)$.
Let $M:=\sup_{W\in \overline{D^{\mathrm{I}}_{2,2}}} |p(W)|$.
Then, we have $M<+\infty$ by the boundedness of $D^{\mathrm{I}}_{2,2}$.
Moreover, by the maximum modulus principle we actually have $|p(W)|<M$ for any $W\in D^{I}_{2,2}$ because $p$ is a non-constant holomorphic function on the bounded domain $D^{\mathrm{I}}_{2,2}$.
We define $h(W):={1\over M} p(W)$.
Then, for any $W\in D^{\mathrm{I}}_{2,2}$ we have
$|h(W)|={1\over M} |p(W)| < 1$.
Thus, $h:D^{\mathrm{I}}_{2,2}\to \mathbb C$ is a holomorphic function  such that for any $W\in D^{\mathrm{I}}_{2,2}$ we have $|h(W)|<1$.
In general, we may replace the polynomial $p(W)$ by any non-constant bounded holomorphic function on $D^{\mathrm{I}}_{2,2}$ in the above.
\end{Example}

In analogy to Lemma \ref{lem:Pro2.3_Ng15}, Seo \cite{Seo18} obtained the following result.

\begin{lemma}[cf.\,Corollary 2.3 in \cite{Seo18}]\label{lem:Cor2.3_Seo18}
Let $D$ and $\Omega$ be irreducible bounded symmetric domains such that $\mathrm{rank}(D)\ge \mathrm{rank}(\Omega)$.
We also let $F: D \times U \to \Omega$ be a holomorphic map such that $F|_{D\times\{w\}}: D \cong D\times\{w\} \to \Omega$ is a proper holomorphic map for each $w\in U$, where $U\Subset \mathbb C^m$ is a connected bounded domain. Then, $f$ does not depend on $w\in U$.
\end{lemma}

For any (reducible) bounded symmetric domain $U$, we write
\[ R_U:=\{\mathrm{rank}(U'): U'\text{ is an irreducible factor of $U$}\} \]
and we define $r_{\mathrm{min}}(U):=\min R_U$ and $r_{\mathrm{max}}(U):=\max R_U$.
We remark here that there are reducible bounded symmetric domains $D$ and $\Omega$ such that $r_{\mathrm{min}}(D)\ge r_{\mathrm{max}}(\Omega)$ and $\mathrm{rank}(D)<\mathrm{rank}(\Omega)$.
For example, for $D=D^{\mathrm{I}}_{3,p_1}\times D^{\mathrm{I}}_{3,p_2}$ and $\Omega = D^{\mathrm{I}}_{3,q_1}\times D^{\mathrm{I}}_{3,q_2} \times D^{\mathrm{I}}_{3,q_3}$, where $p_1,p_2,q_1,q_2,q_3\ge 3$ are integers, we have $r_{\mathrm{min}}(D)=3= r_{\mathrm{max}}(\Omega)$ but $\mathrm{rank}(D)=6<9=\mathrm{rank}(\Omega)$.

From Example \ref{Eg:2}, there is a semi-product proper holomorphic map $f:D\to \Omega$ which is nonstandard and not a product map even if $r_{\mathrm{min}}(D) = r_{\mathrm{max}}(\Omega)-1$ and that $D$ and $\Omega$ have the same number of irreducible factors.
Therefore, for a semi-product proper holomorphic map $f:D\to \Omega$ between bounded symmetric domains $D$ and $\Omega$, by imposing a certain rank condition on $D$ and $\Omega$, namely, $r_{\mathrm{min}}(D) \ge r_{\mathrm{max}}(\Omega)$, we have
\begin{theorem}\label{thm:Structure_PHM1}
Let $D=D_1\times\cdots \times D_k$ and $\Omega=\Omega_1\times\cdots \times \Omega_l$ be bounded symmetric domains, where $D_i$, $1\le i\le k$, and $\Omega_j$, $1\le j\le l$, are irreducible bounded symmetric domains.
Let $f=(f_1,\ldots,f_l):D\to \Omega$ be a semi-product proper holomorphic map.
If $r_{\mathrm{min}}(D)=\min\{\mathrm{rank}(D_i): 1\le i\le k\} \ge \max\{\mathrm{rank}(\Omega_j):1\le j \le l\}=r_{\mathrm{max}}(\Omega)$, then $k\le l$ and we have the following.
\begin{enumerate}
\item Suppose $k=l$.
Then, we have 
\begin{enumerate}
\item[(a)] $\mathrm{rank}(D)=\mathrm{rank}(\Omega)$, $r_{\mathrm{min}}(D)=r_{\mathrm{max}}(\Omega)=:r$ and $\mathrm{rank}(D_i)$ $=$ $\mathrm{rank}(\Omega_j)$ $=$ $r$ for all $i$, $j$, $1\le i\le k$, $1\le j\le l$.
\item[(b)] $f$ is a product map, i.e.,
\[ f(Z^1,\ldots,Z^k) = (f_1(Z^{\sigma(1)}),\ldots,f_k(Z^{\sigma(k)})) \] for some permutation $\sigma\in \Sigma_k$, where $Z^j\in D_j$ for $j=1,\ldots,k$.
\end{enumerate}
\noindent If in addition that $r_{\mathrm{max}}(\Omega)\ge 2$, then $f:D\to \Omega$ is a totally geodesic holomorphic isometric embedding with respect to certain canonical K\"ahler metrics on $D$ and $\Omega$.
\item Suppose $k<l$.
Then, up to a permutation of the irreducible factors $\Omega_j$, $1\le j\le l$, of $\Omega$, we have
\[ f(Z^1,\ldots,Z^k)=(f_1(Z^1),\ldots,f_k(Z^k),f_{k+1}(Z),\ldots,f_l(Z)) \]
for $Z=(Z^1,\ldots,Z^k)\in D_1\times\cdots \times D_k=D$, and for each $i$, $1\le i\le k$, we have $\mathrm{rank}(D_i)=r_{\mathrm{min}}(D)=r_{\mathrm{max}}(\Omega)
=\mathrm{rank}(\Omega_i)$ and $f_i:D_i\to \Omega_i$ is a proper holomorphic map.
If in addition that $r_{\mathrm{max}}(\Omega)\ge 2$, then for $1\le i\le k$, $f_i:D_i\to \Omega_i$ is a totally geodesic holomorphic isometric embedding with respect to the Bergman metrics up to a normalizing constant.
\end{enumerate}
\end{theorem}
\begin{proof}
Our method here is inspired by the proof of Proposition 3.4 in Seo \cite{Seo18}.
Since $f$ is a semi-product map, for $1\le i_1<i_2 \le k$, there are $j_1,j_2\in \{1,\ldots,l\}$ such that
$f_{i_\mu,j_\mu,w^{(\mu)}}:D_{i_\mu} \to \Omega_{j_\mu}$ defined by 
\[ f_{i_\mu,j_\mu,w^{(\mu)}}(Z^{i_\mu}) = f_{j_\mu}(w^{(\mu)}_1,\ldots,w^{(\mu)}_{i_\mu - 1}, Z^{i_\mu}, w^{(\mu)}_{i_\mu + 1},\ldots,w^{(\mu)}_k)\]
is a proper holomorphic map
for $w^{(\mu)}=(w^{(\mu)}_1,\ldots,w^{(\mu)}_{i_\mu - 1}, w^{(\mu)}_{i_\mu + 1},\ldots,w^{(\mu)}_k) \in D_1\times\cdots \times \widehat{D_{i_\mu}} \times \cdots \times D_k$, $\mu=1,2$.
Here, $\widehat{D_{i_\mu}}$ means that the factor $D_{i_\mu}$ is omitted.
If $j_1=j_2$, then
\[ f_{j_1}(w_1,\ldots,w_{i_1 - 1}, \cdot, w_{i_1 + 1},\ldots,w_{i_2-1},\cdot,w_{i_2+1}, \ldots,w_k): D_{i_1}\times D_{i_2} \to \Omega_{j_1} \]
is a proper holomorphic map, a plain contradiction because $\mathrm{rank}(D_{i_1}\times D_{i_2})>\mathrm{rank}(D_{i_1}) \ge \mathrm{rank}(\Omega_{j_1})$ by the assumption (cf. \cite{Ts93}).
Thus, we have $j_1\neq j_2$.
In particular, for any $i\in \{1,\ldots,k\}$, there exists $n_i\in \{1,\ldots,k\}$ such that $f_{i,n_i,w}:D_i \to \Omega_{n_i}$ is a proper holomorphic map for $w\in D_1\times \cdots \times \widehat{D_i} \times \cdots \times D_k$ and $n_i\neq n_\mu$ whenever $i\neq \mu$.
Then, we have $k\le l$.

Now, we may assume that $n_i=i$ for $i=1,\ldots,k$ after permuting the irreducible factors of $\Omega$.
Note that $\mathrm{rank}(D_i) \ge r_{\mathrm{min}}(D) \ge r_{\mathrm{max}}(\Omega) \ge \mathrm{rank}(\Omega_i)$ for $1\le i\le k$.
Applying Corollary 2.3 in Seo \cite{Seo18} (i.e.,\,Lemma \ref{lem:Cor2.3_Seo18}) to $f_{i,i,w}$ for each $i$, we obtain that $f_i$ depends only on $Z^i\in D_i$ and $f_i:D_i\to \Omega_i$ is a proper holomorphic map for $1\le i\le k$.

\medskip
\noindent\textbf{Case (1)} Suppose $k=l$. Then, $f$ is a product map.
Moreover, we have
\begin{equation}\label{Eq:rk1}
\mathrm{rank}(D) \ge k r_{\mathrm{min}} (D) 
\ge k r_{\mathrm{max}}(\Omega) 
  = l r_{\mathrm{max}}(\Omega) 
\ge \mathrm{rank}(\Omega),
\end{equation}
i.e., $\mathrm{rank}(D) \ge \mathrm{rank}(\Omega)$.
From Tsai \cite[p.\,129]{Ts93}, we have $\mathrm{rank}(D) = \mathrm{rank}(\Omega)$. Thus, each inequality in Eq. (\ref{Eq:rk1}) is actually an equality. In particular, we have 
\[ \mathrm{rank}(D_i)=r_{\mathrm{min}} (D)=r_{\mathrm{max}}(\Omega)=\mathrm{rank}(\Omega_j) \]
for all $i$, $j$, $1\le i\le k$, $1\le j\le l$.

If in addition that $r_{\mathrm{max}}(\Omega)\ge 2$, then we have $\mathrm{rank}(D_i) \ge r_{\mathrm{min}}(D) \ge r_{\mathrm{max}}(\Omega)\ge \mathrm{rank}(\Omega_i) \ge 2$ for $1\le i\le k$. By Tsai \cite[Main Theorem]{Ts93}, $f_i:D_i\to \Omega_i$ is a totally geodesic holomorphic isometric embedding with respect to the Bergman metrics up to a normalizing constant for $1\le i\le k$. Hence, $f:D\to \Omega$ is a totally geodesic holomorphic isometric embedding with respect to certain canonical K\"ahler metrics on $D$ and $\Omega$. (Noting that the result also follows directly from Theorem \ref{thm:TrIrr2} in this situation.)

\medskip
\noindent\textbf{Case (2)} Suppose $k<l$. From the above, we have
\[ f(Z^1,\ldots,Z^k)=(f_1(Z^1),\ldots,f_k(Z^k),f_{k+1}(Z),\ldots,f_l(Z)) \]
after permuting the irreducible factors $\Omega_j$, $1\le j\le k$, of $\Omega$. The rest follows from the arguments in Case (1).

\end{proof}

\begin{remark}\text{}
\begin{enumerate}
\item By Proposition 3.5 in Seo \cite{Seo18} and Theorem \ref{thm:Structure_PHM1}, we know that any rational proper holomorphic map $f$ from $D_1 \times \cdots \times D_k$ to
$\Omega_1\times\cdots \times \Omega_k$ is a product map whenever $\mathrm{rank}(D_i)=\mathrm{rank}(\Omega_j)=r$ for all $i,j$, $1\le i,j\le k$ and $r$ is independent of $i$ and $j$.
\item Define a holomorphic map $f:D^{\mathrm{III}}_3\times D^{\mathrm{III}}_3 \to D^{\mathrm{I}}_{3,q_1}\times D^{\mathrm{I}}_{3,q_2} \times \Delta$ by
\[ f(Z^1,Z^2) := \left( \begin{bmatrix}
Z^1 & {\bf 0}
\end{bmatrix}, \begin{bmatrix}
Z^2 & {\bf 0}
\end{bmatrix}, h(Z^1,Z^2) \right), \]
where $q_1,q_2\ge 3$ are integers, and $h:D^{\mathrm{III}}_3\times D^{\mathrm{III}}_3 \to \mathbb C$ is a holomorphic function such that $|h(Z^1,Z^2)|<1$ on $D^{\mathrm{III}}_3\times D^{\mathrm{III}}_3$.
{\rm(}Noting that $r_{\mathrm{min}}(D^{\mathrm{III}}_3\times D^{\mathrm{III}}_3)=3 = r_{\mathrm{max}}(D^{\mathrm{I}}_{3,q_1}\times D^{\mathrm{I}}_{3,q_2} \times \Delta)$ because $q_1,q_2\ge 3$.{\rm)}
Then, we may choose a function $h$ so that $f:D^{\mathrm{III}}_3\times D^{\mathrm{III}}_3 \to D^{\mathrm{I}}_{3,q_1}\times D^{\mathrm{I}}_{3,q_2} \times \Delta$ is a semi-product proper holomorphic map which is not totally geodesic. That means in Case {\rm(2)} of Theorem \ref{thm:Structure_PHM1}, it is possible that such a semi-product proper holomorphic map is not totally geodesic.
\item By Proposition 3.5 in \cite{Seo18} {\rm(}i.e.,\,Proposition \ref{Pro:Pro3.5_Seo18}{\rm)}, the statement of Theorem \ref{thm:Structure_PHM1} still holds true if we assume that $f$ is rational instead of $f$ is semi-product. In other words, Theorem \ref{thm:Structure_PHM1} gives a complete description of all rational proper holomorphic maps $f:D\to \Omega$ between {\rm (}reducible{\rm )} bounded symmetric domains when $r_{\mathrm{min}}(D)\ge r_{\mathrm{max}}(\Omega)$.
\end{enumerate}
\end{remark}

\begin{center}
\textsc{Acknowledgment}
\end{center}
The author would like to thank the anonymous referee for helpful suggestions.


\begin{thebibliography}{XXXXX}
\bibitem[Ch19]{Ch19} Shan Tai Chan: \emph{Remarks on holomorphic isometric embeddings between bounded symmetric domains},
Complex Anal. Synerg. 5 (2019), 5:7.
\bibitem[CXY17]{CXY17} Shan Tai Chan, Ming Xiao and Yuan Yuan: \emph{Holomorphic maps between products of complex unit balls}, Internat. J. Math. 28 (2017), no. 9, 1740010, 22 pp.
\bibitem[D88]{D88} D'Angelo, John P.: \emph{Proper holomorphic maps between balls of different dimensions}, Michigan Math. J. 35 (1988),83--90.
\bibitem[Mok89]{Mok89} Ngaiming Mok: \emph{Metric rigidity theorems on Hermitian locally symmetric manifolds}, Series in Pure Mathematics, Vol. 6, World Scientific Publishing Co., Singapore; Teaneck, NJ, 1989.
\bibitem[Mok12]{Mok12} Ngaiming Mok: \emph{Extension of germs of holomorphic isometries up to normalizing constants with respect to the Bergman metric}, 
J. Eur. Math. Soc. (JEMS) 14 (2012), 1617--1656.
\bibitem[Mok16]{Mok16} Ngaiming Mok: \emph{Holomorphic isometries of the complex unit ball into irreducible bounded symmetric domains}, Proc. Amer. Math. Soc. 144 (2016), pp. 4515--4525.
\bibitem[MT92]{MT92} Ngaiming Mok and I-Hsun Tsai: \emph{Rigidity of convex realizations of irreducible bounded symmetric domains of rank $\ge 2$}, J. reine angew. Math. 431 (1992), 91--122.
\bibitem[Ng15]{Ng15} Sui-Chung Ng: \emph{On proper holomorphic mappings among irreducible bounded symmetric domains of rank at least $2$}, Proc. Amer. Math. Soc. 143 (2015), 219--225.
\bibitem[Seo18]{Seo18} Aeryeong Seo: \emph{Remark on Proper Holomorphic Maps Between Reducible Bounded Symmetric Domains}, Taiwanese J. Math. 22 (2018), 325--337.
\bibitem[Ts93]{Ts93} I-Hsun Tsai: \emph{Rigidity of proper holomorphic maps between symmetric domains}, J. Differential Geom. 37 (1993), pp. 123--160.
\bibitem[Wo72]{Wo72} Joseph A. Wolf: \emph{Fine structures of Hermitian symmetric spaces}, Symmetric spaces (Short Courses, Washington Univ., St. Louis, Mo., 1969-1970), pp. 271--357. Pure and App. Math., Vol. 8, Dekker, New York, 1972.
\bibitem[YZ12]{YZ12} Yuan Yuan and Yuan Zhang: \emph{Rigidity for local holomorphic isometric embeddings from $\mathbb B^n$ into $\mathbb B^{N_1}\times\cdots \times \mathbb B^{N_m}$ up to conformal factors}, J. Differential Geometry \textbf{90} (2012) pp. 329--349.
\end{thebibliography}
\end{document}